\documentclass[12pt]{article}
\usepackage[margin=1in]{geometry}
\usepackage{float,graphicx,subcaption}
\usepackage{enumitem,nameref}
\usepackage{hyperref}
\usepackage{gokulsty}
\allowdisplaybreaks

\makeatletter
\let\orgdescriptionlabel\descriptionlabel
\renewcommand*{\descriptionlabel}[1]{%
  \let\orglabel\label
  \let\label\@gobble
  \phantomsection
  \edef\@currentlabel{#1\unskip}%
  \let\label\orglabel
  \orgdescriptionlabel{#1}%
}
\makeatother

\providecommand{\keywords}[1]
{
  \small	
  \textbf{Keywords---} #1\smallskip
  
}

\providecommand{\msc}[1]
{
  \small	
  \textbf{Mathematics Subject Classification (2020) } #1\smallskip
}

\title{Energy Minimizing Configurations for Single-Director Cosserat Shells}
\author{Timothy J. Healey\\
\small{Department of Mathematics, Cornell University}\and
Gokul G. Nair\footnote{\texttt{gn234@cornell.edu}}\\
\small{Center for Applied Mathematics, Cornell University}}
\date{}

\begin{document}

\maketitle
\begin{abstract}
We consider a class of single-director Cosserat shell models accounting for both curvature and finite mid-plane strains. We assume a polyconvexity condition for the stored-energy function that reduces to a physically correct membrane model in the absence of bending.  With appropriate growth conditions, we establish the existence of energy minimizers. The local orientation of a minimizing configuration is maintained via the blowup of the stored energy as a version of the local volume ratio approaches zero. Finally, we specialize our results to three constrained versions of the theory commonly employed in the subject.
\end{abstract}
\noindent
\noindent
\msc{74B20, 35D99, 49K20}

\section{Introduction}\label{sec:introduction}
We consider a class of single-director Cosserat shell models, accounting for both curvature and finite mid-plane strains.  We are motivated in part by the wrinkling of highly stretched, thin elastomer sheets, e.g.,~\cite{li2016stability},~\cite{nayyar2011stretch}.  In addition, shell models incorporating finite mid-plane strains are widely available in commercial codes, e.g., as employed in~\cite{nayyar2011stretch}.  However, most rigorous existence results for shells rely on small-thickness expansions from 3D elasticity, yielding theories characterized by small (or even zero) mid-plane strains, e.g.,~\cite{friesecke2006hierarchy}.  These are inadequate for predicting the phenomena we have in mind.  Instead, we provide an existence theorem for a general class of direct shell models, incorporating a physically correct mid-plane or membrane energy.  Local orientation of configurations is maintained via the blowup of the energy density function as a version of the local volume ratio approaches zero.  As such, we obtain global energy minimizers.  However, as in the case of bulk nonlinear elasticity, we are generally unable to ensure that these correspond to weak solutions of the Euler-Lagrange equilibrium equations.
\smallskip

A general approach to polyconvexity in the context of higher-order elasticity problems is presented in~\cite{ball1981null}. Among other examples, the set-up for single-director Cosserat continua is outlined there. On that basis, the complete details for energy minimization in single-director Cosserat shells are presented in~\cite{ciarlet2013orientation}. Local orientation preservation of minimizers is carried out in the same manner indicated above. However, the model in~\cite{ciarlet2013orientation} includes a “relaxed” mid-plane energy density.  That is, in the absence of bending energy, the reduced membrane density maintains polyconvexity in the sense of~\cite{ball1981null}. Apparently, this is equivalent to some type of tension-field theory~\cite{pipkin1986relaxed}, and the inclusion of bending energy is somewhat redundant.  In any case, the model precludes wrinkling.  Here we present a modified definition of polyconvexity such that, in the absence of bending energy, the membrane energy density is not rank-one convex (much less polyconvex).  At the same time, our membrane energy is polyconvex when restricted to planar 2D nonlinear elasticity. Among other things, the incorporation of bending energy in this model, enables the resolution of wrinkling patterns. We elaborate on this in Section~\ref{sec:concluding-remarks}.
\smallskip

The outline of the work is as follows: We formulate the class of problems considered and state our hypotheses in Section~\ref{sec:problem_formulation}. The latter includes growth conditions, a polyconvexity condition distinct from that employed in~\cite{ciarlet2013orientation}, and the blow-up of the stored-energy function as a measure of the local volume ratio approaches zero. In Section~\ref{sec:weak-lower-semi-continuity}, we establish weak lower semicontinuity of the energy functional, and we prove our existence theorem in Section~\ref{sec:energy-minimizers}. We specialize our results to three constrained versions of single-director theories in Section~\ref{sec:constrained-minimizers}: (1) the so-called special theory characterized by a unit director field (shearable without thickness change); (2) the director field is normal to the surface (unshearable with thickness change); (3) the classical Kirchhoff-Love hypothesis whereby the director field coincides with a unit normal field on the surface (unshearable without thickness change). We make some concluding remarks in Section 6.

\section{Problem Formulation}\label{sec:problem_formulation}
We let $\RR^n$ denote both Euclidean point space and its translate or tangent space, and we henceforth make the identification $\RR^2\cong\text{span}\{\bm e_1,\bm e_2\}$, where $\{\bm e_1,\bm e_2,\bm e_3\}$ denotes the standard orthonormal basis for $\RR^3$. Let $\Omega\subset\RR^2$ be an open bounded domain with a strongly locally Lipschitz boundary $\partial\Omega$. We associate $\bar\Omega$ with a reference configuration for a material surface in a ``flat'' state as follows: A \textit{configuration} is specified by two fields on $\bar\Omega$: a deformation $\bm f:\bar\Omega\rightarrow\RR^3$ and a director field $\bm d:\bar\Omega\rightarrow\RR^3$; the reference configuration corresponds to $\bm f(x)\equiv x$ and $\bm d(x)\equiv\bm e_3$. The latter need not be stress free. The gradients or total derivatives of $\bm f,\bm d$ at $x\in\Omega$ are denoted $\bm F(x):=\nabla \bm f(x)$, $\bm G(x):=\nabla\bm d(x)\in L(\RR^2,\RR^3)$, respectively. We further require that a smooth configuration satisfy the local orientation condition
\begin{equation}\label{eqn:orientation-preservation}
    J(\bm d,\nabla\bm f):=\bm d\cdot (\bm f_{,1}\times\bm f_{,2})>0\text{ in }\bar\Omega,
\end{equation}
where $\bm f_{,\alpha}$, $\alpha=1,2$, denote partial derivatives and $\bm a\times\bm b$ is the usual right-handed cross product in $\RR^3$. We list the set of 15 independent 2$\times$2 sub-determinants, $\{m_l(\bm F,\bm G)\}_{l=1}^{15}$, of the 2$\times$6 gradient matrix $\left[\bm F^T\ |\ \bm G^T\right]$, which play an important role in what follows:
\begin{equation*}
        \begin{aligned}[c]
            m_{1} &= F_{21}F_{32} - F_{22}F_{31}\\
            m_{4} &= F_{11}G_{12} - F_{12}G_{11}\\
            m_{7} &= F_{21}G_{12} - F_{22}G_{21}\\
            m_{10} &= F_{31}G_{12} - F_{32}G_{11}\\
            m_{13} &= G_{21}G_{32} - G_{22}G_{31}\\
        \end{aligned}
        \qquad
        \begin{aligned}
            m_{2} &= F_{31}F_{12} - F_{32}F_{11}\\
            m_{5} &= F_{11}G_{22} - F_{12}G_{21}\\
            m_{8} &= F_{21}G_{22} - F_{22}G_{21}\\
            m_{11} &= F_{31}G_{22} - F_{32}G_{21}\\
            m_{14} &= G_{12}G_{31} - G_{11}G_{32}\\
        \end{aligned}
        \qquad
        \begin{aligned}[c]    
            m_{3} &= F_{11}F_{22} - F_{12}F_{21}\\
            m_{6} &= F_{11}G_{32} - F_{12}G_{31}\\
            m_{9} &= F_{21}G_{32} - F_{22}G_{31}\\
            m_{12} &= F_{31}G_{32} - F_{32}G_{31}\\
            m_{15} &= G_{11}G_{22} - G_{12}G_{21}.\\
        \end{aligned}
    \end{equation*}

%When the context is clear, we will often refer to $J(\bm f,\bm d)$ simply as $J$.

We assume that the surface is equipped with a stored-energy function, $W(x,\bm d,\bm F,\bm G)$, $W:\bar\Omega\times \mathcal{O}^+\times\RR^{3\times 2}\rightarrow [0,\infty)$, satisfying objectivity:
\begin{align*}
    W(x,\bm Q\bm d,\bm Q\bm F,\bm Q\bm G)\equiv W(x,\bm d,\bm F,\bm G)\quad\text{for all }\bm Q\in SO(3),
\end{align*}
where $\mathcal{O}^+:=\left\{(\bm d,\bm F)\in\RR^3\times\RR^{3\times 2}:J>0\right\}$.
\smallskip

\noindent
We further assume:
\begin{description}
    \item[(H1)\label{itm:growth}] There exist constants $p,q,r>4/3$, $s>1$, $C_1>0$ and $C_2\in\RR$ such that
    \begin{equation*}
        W(x,\bm d,\bm F,\bm G)\geq C_1\left\{\abs{\bm F}^p +\abs{\bm G}^q+\sum_{l=1}^3\abs{m_l}^r+\sum_{l=4}^{15}\abs{m_l}^s\right\} +C_2.
    \end{equation*}
    
    \item[(H2)\label{itm:polyconvexity}] There is a $C^1$ function $\Phi:\bar\Omega\times\RR^3\times\RR^{3\times2}\times\RR^{3\times 2}\times(0,\infty)\times\RR^{12}\rightarrow[0,\infty)$, such that
    \begin{align}\label{eqn:polyconvex}
        (\bm F,\bm G,J,m_4,...,m_{15})\mapsto \Phi(x,\bm d,\bm F,\bm G,J,m_4,...,m_{15})\text{ is convex}
    \end{align}
    and $W(x,\bm d,\bm F,\bm G)\equiv\Phi(x,\bm d,\bm F,\bm G,J,m_4,...,m_{15})$.
    
    \item[(H3)\label{itm:growth2}] $\Phi\rightarrow +\infty$ as $J\rightarrow0^+.$
\end{description}
\smallskip

Let $L^p(\Omega,\RR^3)$ denote the space of $L^p$-integrable 3-vector valued functions on $\Omega$ and let $W^{1,p}(\Omega,\RR^3)\subset L^p(\Omega,\RR^3)$ denote the Sobolev space of vector fields whose weak partial derivatives are also $L^p$-integrable. When the context is clear, we will avoid writing the domain and co-domain in our notation and simply refer to these spaces as $L^p$ and $W^{1,p}$. The norms on these spaces are defined by
\begin{align*}
    &\norm{\bm f}^p_{L^p(\Omega,\RR^3)}=\int_\Omega\abs{\bm f}^p\dif x,\\
    &\norm{\bm f}^p_{W^{1,p}(\Omega,\RR^3)}=\norm{\bm f}^p_{L^p(\Omega,\RR^3)}+\int_\Omega\abs{\nabla\bm f}^p\dif x.
\end{align*}
Consider a subset $\Gamma\subset\partial\Omega$ with positive length, i.e., $\abs{\Gamma}_{\partial\Omega}>0$, and define
\[
    W_\Gamma^{1,p}(\Omega,\RR^3)=\{\bm u\in W^{1,p}(\Omega,\RR^3):\bm u = 0\text{ a.e.~on }\Gamma\},
\]
where $\bm u$ on the boundary is understood in the sense of trace. We define the \textit{admissible set}

\begin{align*}
    \mathcal{A}:=\{(\bm f,\bm d)\in W^{1,p}(\Omega,\RR^3)\times W^{1,q}(\Omega,\RR^3):m_l\in L^r(\Omega),\, l=1,2,3;&\\
    m_l\in L^s(\Omega),\, l=4,...,15;\, J\in L^1(\Omega);\, J>0 \text{ a.e.~in }\Omega;&\\
    \bm f -\bm f_o\in W^{1,p}_\Gamma(\Omega,\RR^3);\,\bm d- \bm d_o\in W^{1,q}_\Gamma(\Omega,\RR^3)&\},
\end{align*}
where $(\bm f_o,\bm d_o)\in W^{1,p}(\Omega,\RR^3)\times W^{1,q}(\Omega,\RR^3)$ are prescribed and satisfy $\bm d_o\cdot(\bm f_{o ,1}\times\bm f_{o, 2})>0$ a.e. 
\smallskip

\begin{remark}\label{remark:p>2}
    We note that when $p,q>2$, a weakened version of~\ref{itm:growth}, viz.,
    \begin{equation*}
        W(x,\bm d,\bm F,\bm G)\geq C_1\left\{\abs{\bm F}^p +\abs{\bm G}^q\right\} +C_2,
    \end{equation*}
    is sufficient to establish the results that follow. In this case, the requirement $m_l\in L^r$ (or $L^s$) can be dropped from the definition of $\mathcal{A}$ as well.
\end{remark}

The total potential energy is given by
\begin{equation}\label{eqn:energy}
    E[\bm f,\bm d]=\int_\Omega W(x,\bm d(x),\nabla\bm f(x),\nabla\bm d(x))\dif x-L(\bm f,\bm d),
\end{equation}
where $L$ is a bounded linear functional on $W^{1,p}(\Omega,\RR^3)\times W^{1,q}(\Omega,\RR^3)$ representing ``dead'' loading. For example,
\begin{equation}
    L(\bm f,\bm d)=\int_\Omega(\bm b\cdot\bm f+\bm g\cdot \bm d)\dif x+\int_{\Gamma^c}[\bm \tau\cdot\bm f+\bm \mu\cdot \bm d]\dif s,
\end{equation}
where $\bm b,\bm g\in L^\infty(\Omega,\RR^3)$, $\bm\tau,\bm\mu\in L^\infty(\Gamma^c,\RR^3)$ are prescribed loadings and $\Gamma^c:=\partial\Omega\setminus\Gamma$.

\section{Weak Lower Semicontinuity}\label{sec:weak-lower-semi-continuity}
We show that~\ref{itm:polyconvexity} implies weak lower semicontinuity of $E[\cdot]$ in the following sense:
\begin{proposition}\label{prop:wlsc}
    The energy functional~\eqref{eqn:energy} is weakly lower semicontinuous, i.e.,
    \begin{align*}
        \liminf_{k\rightarrow\infty}E[\bm f^k,\bm d^k]\geq E[\bm f,\bm d],
    \end{align*}
    whenever $\bm f^k\weakarrow\bm f$ weakly in $W^{1,p}$, $\bm d^k\weakarrow \bm d$ weakly in $W^{1,q}$, $J^k:=J(\bm d^k,\nabla\bm f^k)\weakarrow J:=J(\bm d,\nabla\bm f)$ with $J^k,J>0$ a.e.,~and $m_l^k:=m_l(\nabla\bm f^k,\nabla\bm d^k)\weakarrow m_l:=m_l(\nabla\bm f,\nabla\bm d)$, weakly in $L^1$, $l=4,...,15$, for $p,q\geq 1$.
\end{proposition}
\begin{proof}
    Since $L$ is weakly continuous, we focus on the internal energy
    \[
        I[\bm f,\bm d]:=\int_\Omega W(x,\bm d,\nabla\bm f,\nabla\bm d)\dif x.
    \]
    Assume (by passing a subsequence, if necessary) that 
    \[
        \lim_{k\rightarrow\infty}I[\bm f^k,\bm d^k]=\liminf_{k\rightarrow\infty}I[\bm f^k,\bm d^k].
    \]
    %Let $\Omega_\epsilon:=\{x\in\Omega:\abs{\bm d}+\abs{\nabla\bm f}+\abs{\nabla\bm d}\leq 1/\epsilon,\ J(x)\geq\epsilon>0\}$. 
    From compact embedding~\cite{brezis2010functional}, $\bm d^k\weakarrow\bm d$ in $W^{1,q}\implies\bm d^k\rightarrow \bm d$ strongly in $L^q$. Consequently, for some subsequence (without relabeling) $\bm d^k\rightarrow\bm d$ pointwise a.e. By Egorov's theorem, for $\epsilon>0$, there is a set $\mathcal{U}_\epsilon$ with $\abs{\Omega\setminus\mathcal{U}_\epsilon}\leq\epsilon$ such that $\bm d^k\rightarrow\bm d$ uniformly on $\mathcal{U}_\epsilon$. We also define $\Upsilon_{\epsilon}:=\{x\in\Omega:\abs{\bm d}+\abs{\nabla\bm f}+\abs{\nabla\bm d}\leq 1/\epsilon,\ J\geq\epsilon\}$ and $\Omega_\epsilon:=\mathcal{U}_\epsilon\cap\Upsilon_\epsilon$. Hence, $\abs{\Omega\setminus\Omega_\epsilon}\rightarrow0$ as $\epsilon\rightarrow 0$. By virtue of~\ref{itm:polyconvexity}, we then find
    \begin{align}\label{eqn:hyperplane-applied-to-polyconvex}
        \int_\Omega W(x,\bm d^k,\nabla\bm f^k,\nabla\bm d^k)\dif x\geq& \int_{\Omega_{\epsilon}} W(x,\bm d^k,\nabla\bm f^k,\nabla\bm d^k)\dif x\nonumber\\
        =&\int_{\Omega_{\epsilon}}\Phi( x,\bm d^k,\nabla \bm f^k,\nabla\bm d^k,J^k,m_4,...,m_{15})\dif x\nonumber\\
        \geq&\int_{\Omega_{\epsilon}}\Phi( x,\bm d^k,\nabla\bm f,\nabla\bm d,J,m_4,...,m_{15})\dif x\nonumber\\ 
        &+ \int_{\Omega_{\epsilon}}\Dif_{\bm F}\Phi( x,\bm d^k,\nabla\bm f,\nabla\bm d,J,m_4,...,m_{15})\cdot(\nabla\bm f^k -\nabla\bm f)\dif x\nonumber\\
        &+ \int_{\Omega_{\epsilon}}\Dif_{\bm G}\Phi( x,\bm d^k,\nabla\bm f,\nabla\bm d,J,m_4,...,m_{15})\cdot(\nabla\bm d^k -\nabla\bm d)\dif x\nonumber\\
        &+ \int_{\Omega_{\epsilon}}\Dif_{J}\Phi( x,\bm d^k,\nabla\bm f,\nabla\bm d,J,m_4,...,m_{15})(J^k-J)\dif x\nonumber\\
        &+ \int_{\Omega_{\epsilon}}\sum_{l=4}^{15}\Dif_{m_l}\Phi( x,\bm d^k,\nabla\bm f,\nabla\bm d,J,m_4,...,m_{15})(m_l^k-m_l)\dif x,
    \end{align}
    where, $\Dif_\rho$, $\rho=\bm F,\bm G,J,m_4,...,m_{15}$ denotes the partial derivatives.
    
    In the limit $k\rightarrow\infty$, weak convergence implies that the last four integrals in the final inequality of~\eqref{eqn:hyperplane-applied-to-polyconvex} all vanish, while the first converges to $\int_\Omega\chi_{\Omega_\epsilon}W(x,\bm d,\nabla\bm f,\nabla\bm d)\dif x$, where $\chi_{(\cdot)}$ denotes the characteristic function. Taking $\epsilon\rightarrow 0$, the desired result then follows from the monotone convergence theorem, and hence, $\lim_{k\rightarrow\infty} I[\bm f_k,\bm d_k]\geq\int_\Omega W(x,\bm d,\nabla\bm f,\nabla\bm d)\dif x=I[\bm f,\bm d]$.
\end{proof}

\section{Energy Minimizers}\label{sec:energy-minimizers}
Our main result is the following:
\begin{theorem}\label{thm:existence1}
    Suppose that $\mathcal{A}$ is non-empty with $\inf_{\mathcal{A}}E[\bm f,\bm d]<\infty$. Then there exists $(\bm f^*,\bm d^*)\in\mathcal{A}$ such that $E[\bm f^*,\bm d^*]=\inf_{\mathcal A}E[\bm f,\bm d]$.
\end{theorem}
\begin{proof}
    In this proof, we will focus on the more technical case when $4/3<p,q\leq 2$. Otherwise, the arguments simplify.
    
    Integrating the growth condition~\ref{itm:growth} yields
    \begin{align*}
        \int_\Omega W(x,\bm d,&\nabla\bm f,\nabla\bm d)\dif x\geq C_1\left\{\norm{\nabla\bm f}^p_{L^p}+\norm{\nabla\bm d}^q_{L^q}+\sum_{l=1}^{3}\norm{m_l}_{L^r}^r+\sum_{l=4}^{15}\norm{m_l}_{L^s}^s\right\}+C_2'
    \end{align*}
    A generalized Poincar\'e inequality~\cite{morrey2009multiple} reads
    \[
        \int_\Omega\abs{v}^p\dif x\leq C\left\{\int_\Omega\abs{\nabla v}^p\dif x+\abs{\int_\Gamma Tv\dif a}^p\right\},\quad 1\leq p<\infty,
    \]
    where $T:W^{1,p}(\Omega)\rightarrow L^p(\partial\Omega)$ is the trace operator. From this, we find
    
    \begin{align*}
        \int_\Omega W( x,\bm d,&\nabla\bm f,\nabla\bm d)\dif x\geq C_1'\left\{\norm{\bm f}_{W^{1,p}}^p+\norm{\bm d}^q_{W^{1,q}}+\sum_{l=1}^{3}\norm{m_l}_{L^r}^r+\sum_{l=4}^{15}\norm{m_l}_{L^s}^s\right\}+C_2'',
    \end{align*}
    with constants $C_1'>0$ and $C_2''$. Furthermore, since $L(\bm f,\bm d)$ is a bounded linear functional on $W^{1,p}\times W^{1,q}$ we have
    \[
        \abs{L(\bm f,\bm d)}\leq C_3\left(\norm{\bm f}_{W^{1,p}}+\norm{\bm d}_{W^{1,q}}\right),
    \]
    Since $p,q>1$, the last two inequalities yield
    \begin{equation}
    \label{eqn:coercivity}
        E[\bm f,\bm d]\geq C\left\{\norm{\bm f}_{W^{1,p}}^p+\norm{\bm d}^q_{W^{1,q}}+\sum_{l=1}^{3}\norm{m_l}_{L^r}^r+\sum_{l=4}^{15}\norm{m_l}_{L^s}^s\right\}+D,
    \end{equation}
    where $C>0$ and $D$ are constants.
    \bigskip
    
    Let $\{(\bm f^k,\bm d^k)\}\subset\mathcal A$ be a minimizing sequence for $E[\cdot]$, i.e.,
    \[
        \lim_{k\rightarrow\infty}E[\bm f^k,\bm d^k]=\inf_{(\bm f,\bm d)\in\mathcal{A}}E[\bm f,\bm d].
    \]
    As before, let $m_l^k:=m_l(\nabla \bm f^k,\nabla \bm d^k)$, $l=1,...,15$. By virtue of~\eqref{eqn:coercivity}, we see that the sequences $\{\bm f^k\}$, $\{\bm d^k\}$, $\{m_l^k\}_{l=1}^3$ and $\{m_l^k\}_{l=4}^{15}$ are bounded in $W^{1,p}$, $W^{1,q}$, $L^r$ and $L^s$, respectively, each of which is a  reflexive Banach space. Hence, there exist $\bm f^*\in W^{1,p}$, $\bm d^*\in W^{1,q}$, $\alpha_l\in L^r$ $l=1,2,3$ and $\alpha_l\in L^s$ $l=4,...,15$ and subsequences (not relabeled) converging weakly, i.e.~$\bm f^k\weakarrow\bm f^*$, $\bm d^k\weakarrow\bm d^*$, and $m_l^k\weakarrow\alpha_l$~\cite{brezis2010functional}. For $p=q=2$, the determinants $m^k_l$ $l=1,...,15$ are well-defined $L^1$ functions, while for $4/3<p<2$ and/or $4/3<q<2$, they should be interpreted in the distributional sense, e.g.,
    \begin{align*}
        \int_\Omega \!\!m_4^k\phi\dif x&:=-\frac{1}{2}\!\!\int_\Omega\begin{bmatrix}f^k_{1,1}&-f^k_{1,2}\\-d^k_{1,1}&d^k_{1,2}\end{bmatrix}\begin{bmatrix}f_1^k\\d_1^k\end{bmatrix}\cdot\begin{bmatrix}\phi_{,1}\\\phi_{,2}\end{bmatrix}\dif x\quad\forall\phi\in C^\infty_c(\Omega).
    \end{align*}
    In any case, it is well known that each of these converge as distributions~\cite{dacorogna2007direct}, i.e.,
    \begin{align*}
        \int_\Omega m_l^k\phi\dif x\rightarrow\int_\Omega m_l^*\phi\dif x\quad\forall\phi\in C^\infty_c(\Omega),
    \end{align*}
    where $m_l^*:=m_l(\nabla\bm f^*,\nabla\bm d^*)$, $l=1,...15$. Comparing these to the weak convergence results above, we conclude that
    \begin{align*}
        m_l^k\weakarrow m_l^*\text{ in }L^r,\text{ for }l=1,2,3,\text{ and in }L^s,\text{ for }l=4,...,15. 
    \end{align*}
    
    We now consider the convergence of $J^k:= J(\bm d^k,\nabla\bm f^k)$. We first observe that
    \begin{align}\label{eqn:f_cross-product}
        \bm f_{,1}^k\times\bm f_{,2}^k = m_1^k\bm e_1 + m_2^k\bm e_2 +m_3^k\bm e_3.
    \end{align}
    Thus $\bm f_{,1}^k\times\bm f_{,2}^k\weakarrow \bm f_{,1}^*\times\bm f_{,2}^*$ weakly in $L^r$ for $r>4/3$. In addition, $\bm d^k\weakarrow\bm d^*$ weakly in $W^{1,q}$ implies strong convergence in $L^{q'}$ for $1\leq q'<q^*$ where
    \begin{align*}
        q^*:=
        \begin{cases}
            \frac{2q}{2-q} &1\leq q<2,\\
            \infty &q=2.
        \end{cases}
    \end{align*} 
    Keeping in mind that $q,r>4/3$, if $q\geq r$, we choose $q'=r/(r-1)$ and if $r>q$, we choose $q'=q/(q-1)$. In either case, $\int_\Omega\bm d^k\cdot(\bm f^k_{,1}\times\bm f^k_{,2})\phi\dif x\rightarrow\int_\Omega\bm d^*\cdot(\bm f^*_{,1}\times\bm f^*_{,2})\phi\dif x$ for all $\phi\in L^\infty$, i.e., $J^k\weakarrow J^*$ weakly in $L^1$.
    \smallskip
    
    Next, we show that $(\bm f^*,\bm d^*)\in\mathcal A$. First, we claim that $J^*>0$ a.e. in $\Omega$. By virtue of Mazur's theorem, we can construct a sequence of convex combinations of the sequence $\{J^k\}$ that converges strongly in $L^1$ to $J^*$. Thus, there is a subsequence converging to $J^*$ a.e.~in $\Omega$. Since each $J^k>0$ a.e., we deduce that $J^*\geq0$ a.e. Now suppose that that $J^*=0$ a.e.~in $\mathcal{U}\subset\Omega$, where $\abs{\mathcal{U}}>0$. Employing $\chi_{\mathcal{U}}$ as a test function, the weak convergence of $J^k$ implies $J^k\rightarrow 0$ strongly in $L^1(\mathcal{U})$. Thus, for a subsequence (not relabeled), $J^k\rightarrow 0$ a.e.~in $\mathcal{U}$. But then~\ref{itm:growth2} and Fatou's lemma imply
    \begin{align*}
        \liminf_{j\rightarrow\infty}E[\bm f^j,\bm d^j]\geq\int_{\mathcal{U}}\lim_{j\rightarrow\infty}W(x,\bm d^j(x),\nabla\bm f^j(x),\nabla\bm d^j(x))\dif x +C=\infty,
    \end{align*}
    which contradicts our hypothesis that $\inf_{\mathcal{A}}E[\bm f,\bm d]<\infty$. Hence, $J^*>0$ a.e.~in $\Omega$. In addition, $\bm f^k-\bm f_o\in W^{1,p}_\Gamma(\Omega,\RR^3)$, which is a closed linear subspace of $W^{1,p}$. Thus, $W^{1,p}_\Gamma$ is weakly closed~\cite{brezis2010functional}, and $\bm f^k-\bm f_o\weakarrow\bm f^*-\bm f_o\in W^{1,p}_\Gamma$. Similarly, $\bm d^*-\bm d_o\in W^{1,q}_\Gamma$. We conclude that $(\bm f^*,\bm d^*)\in\mathcal A$.
    
    To complete the proof, we combine the results above with Proposition~\ref{prop:wlsc} to conclude $E[\bm f^*,\bm d^*]\leq\liminf_{k\rightarrow\infty}E[\bm f^k,\bm d^k]$ with $(\bm f^*,\bm d^*)\in\mathcal{A}$, i.e., $E$ attains its infimum on $\mathcal{A}$.
\end{proof}

\section{Constrained Minimizers}\label{sec:constrained-minimizers}
We now explore three different constrained versions of the theory that are common in the study of nonlinearly elastic shells.
\vspace{-0.3cm}
\subsection{Special Theory}\label{sec:special_theory}
Here the director field is constrained to have unit length, i.e.,
\begin{equation}\label{eqn:special-theory}
    \abs{\bm d}=1\text{ a.e.~in }\Omega.
\end{equation}
Again, we assume~\ref{itm:growth}-\ref{itm:growth2} and incorporate~\eqref{eqn:special-theory} into the admissible set:
\begin{align*}
    \mathcal{A}:=\{(\bm f,\bm d)\in W^{1,p}(\Omega,\RR^3)\times W^{1,q}(\Omega,\RR^3):m_l\in L^r(\Omega),\,l=1,2,3;&\\
    m_l\in L^s(\Omega),\,l=4,...,15;\,
    J\in L^1(\Omega);\\
    J>0 \text{ a.e.~in }\Omega;\,\abs{\bm d}=1\text{ a.e.~in }\Omega;\\
    \bm f -\bm f_o\in W^{1,p}_\Gamma(\Omega,\RR^3);\,
    \bm d- \bm d_o\in W^{1,q}_\Gamma(\Omega,\RR^3)&\},
\end{align*}
where $(\bm f_o,\bm d_o)\in W^{1,p}(\Omega,\RR^3)\times W^{1,q}(\Omega,\RR^3)$ are prescribed and $\abs{\bm d_o}=1$ and $\bm d_o\cdot(\bm f_{o,1}\times\bm f_{o,2})>0$ a.e.

The existence of a minimizer follows precisely as before in Theorem~\ref{thm:existence1}. We only need to show that~\eqref{eqn:special-theory} is satisfied. This follows from compact embedding: For a minimizing sequence, we have $\bm d^k\weakarrow\bm d^*$ in $W^{1,r}\implies\bm d^k\rightarrow\bm d^*$ in $L^q$. Thus there is a convergent subsequence $\bm d^{k_n}\rightarrow\bm d^*$ a.e. Since $\abs{\bm d^{k_n}}=1$ a.e., we have $\abs{\bm d^*}=1$ a.e.

\subsection{Normal Director Field}\label{sec:normal_director_field}
We now constrain the director field to be normal to the surface (allowing its length to be variable), viz.,
\begin{equation}\label{eqn:normal-director}
    \bm d\cdot\bm f_{,\alpha}=0\text{ a.e.~in }\Omega,\quad\alpha=1,2.
\end{equation}
Again, we assume~\ref{itm:growth}-\ref{itm:growth2} and incorporate~\eqref{eqn:normal-director} into the admissible set:
\begin{align*}
    \mathcal{A}:=\{(\bm f,\bm d)\in W^{1,p}(\Omega,\RR^3)\times W^{1,q}(\Omega,\RR^3):m_l\in L^r,\,l=1,2,3;&\\
    m_l\in L^s(\Omega),\,l=4,...,15;\,
    J\in L^1(\Omega);\,J>0 \text{ a.e.~in }\Omega&;\\
    \bm d\cdot\bm f_{,\alpha}=0\text{ a.e.~in }\Omega,\,\alpha=1,2&;\\
    \bm f -\bm f_o\in W^{1,p}_\Gamma(\Omega,\RR^3);\,\bm d-\bm d_o\in W^{1,q}_\Gamma(\Omega,\RR^3)&\},
\end{align*}
where $(\bm f_o,\bm d_o)\in W^{1,p}(\Omega,\RR^3)\times W^{1,q}(\Omega,\RR^3)$ are prescribed with the property that $\bm d_o\cdot\bm f_{o,\alpha}=0$ for $\alpha=1,2$ and $\bm d_o\cdot(\bm f_{o,1}\times\bm f_{o,2})>0$ a.e.

Existence of a minimizer again follows as before. In this case, we only need to show that $(\bm f^*,\bm d^*)$ satisfies~\eqref{eqn:normal-director}. Without loss of generality, suppose $p\leq q$, then we have $\bm f^k_{,\alpha}\weakarrow\bm f^*_{,\alpha}$ weakly in $L^p$ and by compact embedding $\bm d^k\rightarrow\bm d^*$ in $L^{\frac{p}{p-1}}$. The case of $p>q$ works similarly.
By admissibility, $0=\int_\Omega\bm d^k\cdot\bm f^k_{,\alpha}\phi\dif x\rightarrow\int_\Omega\bm d^*\cdot\bm f_{,\alpha}^*\phi\dif x$ for all $\phi\in C^\infty_c(\Omega)$, $\alpha=1,2$. By the theorem of DuBois-Reymond~\cite{morrey2009multiple}, we conclude that $(\bm f^*,\bm d^*)$ satisfies~\eqref{eqn:normal-director}.

\subsection{Kirchhoff-Love Theory}\label{sec:Kirchhoff-Love}
In this classical case, the director field is required to coincide with the unit normal field to the deformed surface. This is usually referred to as the Kirchhoff-Love theory.

The constraints are now
\[
    \bm d\cdot \bm f_{,\alpha}=0\quad\text{a.e.~in }\Omega\quad \alpha=1,2;
\]
\[
    \abs{\bm d}=1\quad\text{a.e.~in }\Omega.
\]
We modify the admissible set to accommodate these:
\begin{align*}
    \mathcal{A}:=\{(\bm f,\bm d)\in W^{1,p}(\Omega,\RR^3)\times W^{1,q}(\Omega,\RR^3):m_l\in L^r(\Omega),\,l=1,2,3&;\\
    m_l\in L^s(\Omega),\,l=4,...,15;\,
    J\in L^1(\Omega);\,J>0 \text{ a.e.~in }\Omega&;\\
    \bm d\cdot\bm f_{,\alpha}=0 \text{ a.e.~in }\Omega,\, \alpha=1,2;\,
    \abs{\bm d}=1\text{ a.e.~in }\Omega&;\\
    \bm f-\bm f_o\in W^{1,p}_\Gamma(\Omega,\RR^3);\,
    \bm d-\bm d_o\in W^{1,q}_\Gamma(\Omega,\RR^3)&\},
\end{align*}
where $(\bm f_o,\bm d_o)\in W^{1,p}\times W^{1,q}$ are prescribed and satisfy $\bm d_o\cdot\bm f_{o,\alpha}=0$ for $\alpha=1,2$, $\abs{\bm d_{o}}=1$ and $\bm d_{o}\cdot(\bm f_{o,1}\times\bm f_{o,2})>0$ a.e.

The existence of a minimizer follows as before, and the arguments used in Section~\ref{sec:special_theory} and~\ref{sec:normal_director_field} imply $\abs{\bm d^*}=1$ a.e. and $\bm d^*\cdot\bm f^*_{,\alpha}=0$ a.e., $\alpha=1,2$, respectively.

\section{Concluding Remarks}\label{sec:concluding-remarks}
In our notation, the stored-energy function employed in~\cite{ciarlet2013orientation} takes the form $W(x,\bm d,\bm F,\bm G)=\Psi(x,\bm d,\bm F,\bm G,m_1,...,m_{15})$, where $(\bm F,\bm G,m_1,...,m_{15})\mapsto\Psi$ is convex. In the absence of bending energy, the membrane energy depends only on the deformation gradient, i.e., $\tilde W(x,\bm F)=\tilde\Psi(x,\bm F,m_1,m_2,m_3)$ with $(\bm F,m_1,m_2,m_3)\mapsto\tilde\Psi$ convex. That is, $\tilde W$ is polyconvex as in~\cite{ball1981null}. As noted in Section~\ref{sec:introduction}, this is apparently equivalent to some kind of tension-field theory, i.e., compression is replaced by zero stress. In contrast, the membrane version of~\eqref{eqn:polyconvex} reduces to $\hat{W}(x,\bm F)=\hat{\Phi}(x,\bm F,J)$, with $(\bm F,J)\mapsto\hat{\Phi}$ convex. Here, $J$ denotes the local area ratio of the surface, i.e.,
\begin{align*}
  J:=\bm n\cdot(\bm f_{,1}\times\bm f_{,2})=\abs{\bm f_{,1}\times\bm f_{,2}}=[\det(\bm F^T\bm F)]^{1/2},
\end{align*}
where $\bm n$ denotes the unit normal field in the same direction as $\bm f_{,1}\times\bm f_{,2}$. Observe that $\hat{W}$ is polyconvex for planar deformations only, viz., $\bm F\in\RR^{2\times 2}$. Indeed, it can be shown that $\hat W$ is not even rank-one convex for $\bm F\in\RR^{3\times 2}$. Among other things, this is the correct model for predicting wrinkling: In the absence of bending energy, arbitrarily finer and finer spatial oscillations (wrinkles) are allowed to develop in lieu of sustained compression. Such behavior, in turn, is penalized by bending energy, enabling the resolution of wrinkling amplitudes and wavelengths. Also, compressive stresses are not generally zero.
\smallskip

A rigorous existence theorem for the Kirchhoff-Love model, based on energy minimization, is presented in~\cite{anicic2018polyconvexity}. Local orientation preservation is maintained via a blow-up argument similar to that employed here and in~\cite{ciarlet2013orientation}, except that the volume ratio employed in~\cite{anicic2018polyconvexity} takes into account the surface thickness via the Cosserat ansatz, cf.~\cite{antman2005problems}. The definition of polyconvexity used in ~\cite{anicic2018polyconvexity}, involving only $\bm G,\bm F$ and the volume ratio just described, is much more restrictive than that employed in this work. However, when reduced to its membrane part (ignoring thickness) it agrees with ours. Also, in contrast to our approach based on constraints, the unit normal field is directly parametrized by the surface deformation in~\cite{anicic2018polyconvexity}. A distinction between those results and those of Section~\ref{sec:Kirchhoff-Love} becomes apparent upon taking a formal first variation (not rigorous): Our Euler-Lagrange equations at a minimizer would involve Lagrange multiplier fields enforcing the constraints (representing transverse shears and through-thickness resultants), whereas the latter would be effectively eliminated from the Euler-Lagrange equations associated with~\cite{anicic2018polyconvexity}.
\smallskip

We also mention that~\cite{healey2022existence} is comparable to our results from Section~\ref{sec:normal_director_field}. In the former, the normal director field (not necessarily unit) is directly parametrized by the surface deformation. This entails a full second-gradient surface theory, while local orientation is preserved in a manner similar to that presented here. For growth conditions on the second gradient with $p>2$, the first variation can be taken rigorously at a minimizer, leading to the weak form of the Euler-Lagrange equations. This follows from the same construction used in~\cite{healey2009injective}. Interestingly, it is not at all clear how to do so based on the results of Section~\ref{sec:normal_director_field}.

\section*{Acknowledgements}
This work was supported in part by the National Science Foundation through grant DMS-2006586, which is gratefully acknowledged.

\bibliographystyle{amsalpha}
\bibliography{master}
\end{document}